\documentclass[preprint,3p , onecolumn]{elsarticle}

\usepackage{amssymb}
 \usepackage{amsthm}
\usepackage{amsfonts,amsmath,latexsym,amssymb, amscd, graphicx, color}

\newtheorem{theorem}{Theorem}
\newtheorem{lemma}{Lemma}
\newtheorem{definition}[theorem]{Definition}
\newtheorem{proposition}{Proposition}
\newtheorem{corollary}{Corollary}
\newdefinition{remark}{Remark}
\newdefinition{example}{Example}
\usepackage{hyperref}

\begin{document}

\begin{frontmatter}



\title{Operator convex functions and their applications}



\author[a]{V. Kaleibary\corref{cor1}}\ead{v.kaleibary@gmail.com}
\author[a]{M. R. Jabbarzadeh}\ead{mjabbar@tabrizu.ac.ir}
\author[b]{Shigeru Furuichi }\ead{furuichi@chs.nihon-u.ac.jp}
\cortext[cor1]{Corresponding author}
\address[a]{Faculty of Mathematical Sciences, University of Tabriz
5166615648, Tabriz, Iran}
\address[b]{Department of Information Science, College of Humanities and Sciences, Nihon University,\\
3-25-40, Sakurajyousui, Setagaya-ku, Tokyo, 156-8550, Japan}

\begin{abstract}
In this paper, we introduce operator geodesically convex  and operator convex-log functions and  characterize some properties of them. Then apply these classes of functions to present several operator Azc\'{e}l and Minkowski type inequalities extending some known results. The concavity counterparts are also considered. 
\end{abstract}

 
\end{frontmatter}



\section{Introduction}
\label{}

It is known that the theory of matrix/operator convex functions introduced by Kraus \cite{K1936} have many important applications in matrix analysis and quantum information and so on. Following this study, significant concepts of convexity have been extended elegantly to Hilbert space operators from scalar cases. The main aim of this paper is to establish an analogue of some convexity properties for operator functions. For this purpose,  we first  briefly review a survey on convex functions and operator convex functions.

\begin{definition} {\bf (\cite{N2000})}
Let an interval $J \subset (0,\infty)$, $a,b \in J$, $v\in[0,1]$ and let the function be $f:J\to (0,\infty)$.
\begin{itemize}
\item[(AA)] The function $f$ is said to be a (usual) convex iff 
$$
f((1-v)a+vb) \leq (1-v)f(a) +vf(b).
$$
\item[(AG)] The function $f$ is said to be a log-convex iff 
$$
f((1-v)a+vb) \leq f^{1-v}(a)f^v(b).
$$
\item[(GA)] The function $f$ is said to be a geodesically convex 
iff 
$$
f(a^{1-v}b^v) \leq (1-v)f(a) + vf(b).
$$
\item[(GG)] The function $f$ is said to be a geometrically convex iff 
$$
f(a^{1-v}b^v) \leq  f^{1-v}(a)f^v(b).
$$
\end{itemize}
If inequalities are reversed, then we have the corresponding types of concave functions.
\end{definition}
\begin{definition} {\bf (\cite{H2019})}  \label{definition1.1}
A function $f: (0, \infty) \rightarrow \mathbb{R}$ is called convex-log if it can be written on the form $f(t)= h (\log t), \;t>0$ where $h: \mathbb{R}\rightarrow \mathbb{R}$ is a convex function.
\end{definition}

We give a remark on the basic properties of the above function.

\begin{remark}
\begin{itemize}
\item[(i)]A convex-log function satisfies the inequality 
\begin{align*} 
f(a^{1-v} b^{v})  \leq f(a) \nabla_ v f(b), \;\;\;\;\; a, b>0,
\end{align*}
for $v\in [0, 1]$. Indeed,
\begin{align*} 
f(a^{1-v} b^{v})  &= h(\log(a^{1-v} b^{v}))  = h((1-v) \log(a)+v \log(b)) \\
 & \leq (1-v) h(\log(a)) + v h(\log(b)) = (1-v) f(a)+ v f(b). 
\end{align*}
So, we can say every convex-log function is a geodesically convex function \cite{H2019}. 
\item[(ii)]For a continuous positive function $f$, if $\log f$ is convex, then it is natural to say $f$ to be a log-convex. If $f$ be an increasing log-convex function, then it is a geometrically convex and so a geodesically convex by  the arithmetic-geometric mean inequality. While every (increasing) convex-log function is a geodesically convex and not necessary a geometrically convex function. There are examples that show the difference between these two classes of functions. For instance, the function $f(t)=t^p, \; p\in \mathbb{R}$ is a convex-log, by letting $h(t)= \exp(pt)$. But it is not a log-convex, since $\log(f(t)) = p \log(t)$ is not convex.
\end{itemize}
\end{remark}

 For a real-valued function $f$ and a self adjoint operator $A\in B(\mathcal{H})$, the value $f(A)$ is understood by means of the functional calculus. 
For each $\alpha \in [0, 1]$ and strictly positive operators $A, B$, $A\nabla_\alpha B= (1-\alpha) A+ \alpha B$, $A!_\alpha B = ((1-\alpha) A^{-1}+ \alpha B^{-1})^{-1}$ and 
$A \sharp_\alpha B = A^{1/2} ( A^{-1/2} B A^{-1/2})^{\alpha} A^{1/2}$ are  the $\alpha$-arithmetic, $\alpha$-harmonic and  $\alpha$-geometric means, respectively. It is known that for any $A, B>0$, we have
$A !_{\alpha}B \leq A \sharp_\alpha B \leq A \nabla_\alpha B$. Some of the above definitions of
convexity have been extended to the operator case as follows.
\begin{definition}  \label{definition1.4}
Let $J$ be an interval of $(0, \infty)$. Let $f$ be a continuous real function on $J$, $A, B $  be strictly positive operators with spectra contained in $J$ and  $v\in[0,1]$.
\begin{itemize}
\item[(i)] The function $f$ is said to be an operator convex iff 
$$
f((1-v)A+vB) \leq (1-v)f(A) +vf(B).
$$
\item[(ii)] The nonnegative function $f$ is said to be an operator log-convex iff 
$$
f((1-v)A+vB) \leq f(A) \sharp_v f(B).
$$
\end{itemize}
\end{definition}

The concept of operator convexity was delicately introduced by Kraus \cite{K1936}. Hiai and Ando in \cite{AH2011} obtained a full characterization of
operator log-convex functions.
Also, a variant of geometrically convexity property is presented in \cite{GK2016} as follows:
$$
f(A\sharp_v B) \leq M ( f(A) \sharp_v f(B)), \;\; M>0.
$$ 
In this note, we
extend the definition of geodesically convex  and convex-log functions to the operator space.  In the second section, we first introduce operator geodesically convex (concave) functions. We present some properties of them and show that the class of such functions is fairly rich. Then we obtain an operator Azc\'{e}l inequality, including operator geodesically convex functions. In the third section, we give the definition of an operator log-convex function and investigate some properties of that. Further, a variant of operator Azc\'{e}l inequality involving operator concave-log functions is given.
The last section is devoted to studying another type of geodesically convex functions which leads to getting some Minkowski type inequalities. The obtained results generalize the corresponding Minkowski and Azc\'{e}l inequalities in \cite{BH2014} and \cite{Mos2011}, respectively.

\section{Operator geodesically convex function}

In 1956, Acz\'{e}l \cite{A1956} proved that if $a_i, b_i (1 \leq i \leq n)$ are positive real numbers such that
$a_1^2 - \sum\limits_{i=2}^n a_i^2 >0 $ and $b_1^2 - \sum\limits_{i=2}^n b_i^2 >0 $, then
\begin{align*}
\left(a_1 b_1 - \sum_{i=2}^n a_i b_i \right)^2 \geq  \left(a_1^2 - \sum_{i=2}^n a_i^2 \right) \left(b_1^2 - \sum_{i=2}^n b_i^2 \right).
\end{align*}
Popoviciu \cite{P1995} presented an exponential extension of Acz\'{e}l's
inequality, so that
 if $p>1, q>1, \frac{1}{p}+\frac{1}{q}=1$, 
$a_1^p - \sum\limits_{i=2}^n a_i^p >0 $, and $b_1^q - \sum\limits_{i=2}^n b_i^q >0 $, then
\begin{align*}
 a_1 b_1 - \sum_{i=2}^n a_i b_i \geq  \left(a_1^p - \sum_{i=2}^n a_i^p \right)^\frac{1}{p} \left(b_1^q - \sum_{i=2}^n b_i^q \right)^\frac{1}{q}.
\end{align*}

 Acz\'{e}l's  and Popoviciu's inequalities were
sharpened and some generalizations and variants of these inequalities  are presented. See \cite{D1994} and  references therein. 
An operator version of the classical Acz\'{e}l inequality  was given in \cite{Mos2011}.  Further, some reverses of  the operator Acz\'{e}l inequality were given in \cite{KF2018} and a variant of them  was provided in \cite{FJK2020}.\\
In this section we  introduce an operator geodesically convex (concave) function and  present an operator  Acz\'{e}l inequality involving this class of functions. 
 
\begin{definition}\label{definition2.1}
Let $J$ be an interval of $(0, \infty)$. A nonnegative continuous function $f$ on $J$ is said to be an operator geodesically convex iff
\begin{align}\label{def_ineq01}
f\left(A\sharp_v B\right) \leq  f(A) \nabla_v f(B),
\end{align}
for strictly positive operators $A,B $ with spectra contained in $J$. The function  $f$  is also said to be an operator geodesically concave iff $-f$ is operator geodesically convex.
\end{definition}
We first aim to show  that the class of functions satisfying
\eqref{def_ineq01} is fairly rich. For this purpose,  the following lemmas are provided. We also recall a continuous real function $f$ defined on an interval 
$J$ is said to be operator monotone, if $A \leq B$ implies $f(A)\leq f(B)$ for all $A, B$ with spectra in $J$.

\begin{lemma}\label{Lemma 2.1}
Let  $f, f_1$ and $f_2$ be nonnegative continuous functions on $J\subseteq (0, \infty)$.
\begin{enumerate}
\item[(i)] If $f$ is  operator monotone and operator convex, then $f$ is  operator geodesically convex.
\item[(ii)]  If $f_1$ is  operator monotone and operator convex and $f_2$ is  operator geodesically convex, then $f_1\circ f_2$ is operator geodesically convex.
\item[(iii)] If $f_1$ and $f_2$ are  two operator geodesically convex functions, then so is $\alpha f_1 + f_2$ for $\alpha >0$.
\end{enumerate}
\end{lemma}
\begin{proof} For strictly positive operators $A,B $ we have the well-known Young inequality $A\sharp_v B\leq A \nabla_v B$. Now, 
$(i)$  clearly holds by the assumptions on $f$ and applying the Young inequality. For $(ii)$ we have
\begin{align*}
f_1 \circ f_2(A\sharp_v B) &= f_1(f_2(A\sharp_v B))\nonumber\\
&\leq f_1(f_2(A)\nabla_v f_2(B))\hspace*{1cm}\text{(op. monotonicity of $f_1$ and \eqref{def_ineq01})}\nonumber\\
&\leq f_1(f_2(A)) \nabla_v f_1(f_2(B))\hspace*{1cm}\text{(op. convexity of $f_1$) }\nonumber\\
&= f_1 \circ f_2(A) \nabla_v f_1 \circ f_2( B).\nonumber
\end{align*}
Now, let $f_1$ and $f_2$ be two operator geodesically convex functions and $\alpha >0$. Then 
\begin{align*}
(\alpha f_1 + f_2)(A\sharp_v B) 
&= \alpha f_1(A\sharp_v B)+ f_2(A\sharp_v B) \nonumber\\
&\leq  \alpha \big(f_1(A)  \nabla_v f_1(B)\big) + \big(f_2(A)  \nabla_v f_2(B)\big)\nonumber\\
& =  \big(\alpha f_1(A)  + f_2(A)\big) \nabla_v \big(\alpha f_1(B)  + f_2(B)\big) \nonumber\\
&=(\alpha f_1 + f_2)(A) \nabla_v (\alpha f_1 + f_2)(B). \nonumber
\end{align*}
That is $\alpha f_1 + f_2$ is geodesically convex function as well.
\end{proof}

\begin{lemma}\label{Lemma 2.3}
Let  $f$ and $g$ be continuous functions from $(0, \infty)$ into itself.
\begin{enumerate}
\item[(i)] If $f(x)$ is an operator geodesically convex function so is $f\left(\frac{1}{x}\right)$.
\item[(ii)] If $g(x)$ be an operator geodesically concave function so is $g\left(\frac{1}{x}\right)$. 
\end{enumerate}
\end{lemma}
\begin{proof}
Let $A$ and $B$ be strictly positive operators. Thanks to the geometric mean property $(A\sharp_v B)^{-1} = A^{-1}\sharp_v B^{-1}$, we have the first result as follows
\begin{align*}
f((A\sharp_v B)^{-1})= f(A^{-1}\sharp_v B^{-1})\leq  f(A^{-1})\nabla_v f(B^{-1}).
\end{align*}
The second one is obtained  similarly. 
\end{proof}
In the above lemma if we let $f$ and $g$ be nonnegative continuous functions from $J\subset (0, \infty)$, we will assume that $ J$ contains both $Sp(A)$ and $Sp(A^{-1}) $, where $Sp(A)$ represents the spectrum of $A$.
The next  theorem presents a
connection between operator geodesically concavity and convexity.
We recall that  $g^*(x):=\dfrac{1}{g\left(\frac{1}{x}\right)}$ is called the adjoint of functin $g$.
\begin{theorem}\label{theorem2.4}
Let $g$ be an operator geodesically concave function. Then the functions $1/g$ and $g^*$ are operator geodesically convex.
 \end{theorem}
\begin{proof}
Let $A, B$ be strictly positive operators. For the operator geodesically concave function $g$  we have 
\begin{align}\label{thm_ineq01}
g(A\sharp_v B)\geq g(A)\nabla_v g(B).
\end{align}
Therefore,
\begin{align*}
(g(A\sharp_v B))^{-1} \leq  (g(A)\nabla_v g(B))^{-1} =g(A)^{-1} !_v g(B)^{-1} \leq  g(A)^{-1} \nabla_v  g(B)^{-1},
\end{align*}
which shows $1/g$  is operator geodesically convex. Combining this result with part $(i)$ of Lemma \ref{Lemma 2.3} easily yields  $g^*$ is operator geodesically convex function. However, in the sequel we give a direct proof providing a refinement inequality. Rewriting the inequality \eqref{thm_ineq01} with the operators $A^{-1}, B^{-1}$ and taking the inverse, we have
\begin{align}\label{e5}
\big(g(A^{-1}\sharp_v B^{-1})\big)^{-1} \leq \big(g(A^{-1}) \nabla_v g(B^{-1})\big)^{-1}.
\end{align}
Hence
\begin{align*}
g^*(A\sharp_v B)
&=\big(g(A^{-1}\sharp_v B^{-1})\big)^{-1}\nonumber\\
 &\leq \big(g(A^{-1}) \nabla_v g(B^{-1})\big)^{-1}\hspace*{1cm}\text{(by the inequality \eqref{e5})}\nonumber\\
 &=\big(g^*(A)^{-1} \nabla_v g^*(B)^{-1}\big)^{-1}\nonumber\\
 &=g^*(A) !_v g^*(B) \nonumber\\
 &\leq g^*(A)\nabla_v g^*(B),\nonumber
\end{align*}
as desired.
\end{proof}
\begin{example}
\begin{enumerate}
\item[(i)]  The simplest example of operator geodesically convex functions is $f(t)=t-a, a\geq 0$ on $(a, \infty)$. For $a=0$, $f(t)=t$  leads to the Young inequality $A\sharp_v B\leq A \nabla_v B$.
\item[(ii)] Another example is  $f(t)=\dfrac{1}{1-t}$ on $(0, 1)$ due to its operator convexity  and operator monotonicity \cite{Choi1974}, hence so is $f(1/t)=\dfrac{t}{1-t}$ on $(0, 1)$.
\item[(iii)] Let $f(t)=\dfrac{1}{t}$ on $(0, \infty)$. Then 
\begin{align*}
f(A\sharp_v B)= (A\sharp_v B)^{-1}= A^{-1}\sharp_v B^{-1}\leq A^{-1}\nabla_v B^{-1}=f(A)\nabla_v  f(B).
\end{align*}
This function  is an instance of  operator geodesically convex ones which is not operator monotone. 
\end{enumerate}

\end{example}

\begin{example}\label{example2.6}
\begin{enumerate}
\item[(i)] Every  operator monotone decreasing and operator concave  function $g$ on $J$ is an operator geodesically concave function.

\item[(ii)] It can be seen that the Young inequality $A\sharp_v B \leq A\nabla_v B$  is equivalent to 
$I-A\sharp_vB \geq (I-A) \nabla_v (I-B)$. This means the function $g(t) = 1-t$ on $(0,1)$ is an operator geodesically concave function. Similarly, $g(t)=a-t$ on $(0, a)$. By applying Theorem \ref{theorem2.4} it is deduced the functions $g(t)^{-1}=\dfrac{1}{a-t}$ on $(0, a)$ and $g^*(t)=\dfrac{t}{at-1}$ on $\left(\dfrac{1}{a},\infty\right) $ are operator geodesically convex.
\item[(iii)] Let $g(t)=a-\dfrac{1}{t}$, $t \in \left(\dfrac{1}{a}, \infty\right)$. Then
\begin{align*}
g(A\sharp_v B)= aI-(A\sharp_v B)^{-1}&= aI - (A^{-1}\sharp_v B^{-1})\nonumber\\
&\geq (aI-A^{-1})\nabla_v (aI-B^{-1})\nonumber\\
&=g(A)\nabla_v  g(B).
\end{align*}
Hence, $g(t)$ is operator geodesically concave  which in not operator monotone decreasing. 
\end{enumerate}
\end{example}
The corresponding results of Lemma \ref{Lemma 2.1}  hold for  operator geodesically concave functions as well. The next result provides an operator Acz\'{e}l  inequality involving this class of functions.

\begin{theorem} \label{theorem2.1}
Let $J$ be an interval of $(0,\infty)$, let  $g: J \to [0,\infty)$ be an operator geodesically concave function, and $p,q>1$ with $1/p+1/q=1$. For strictly operators $A$ and $B$ with spectra contained in $J$, we have
\begin{equation}\label{theorem2.1_ineq01}
g\left(A^p\sharp_{1/q} B^q\right) \geq   g(A^p)\sharp_{1/q} g(B^q).
\end{equation}
and
\begin{equation}\label{theorem2.1_ineq02}
\langle g\left(A^p\sharp_{1/q} B^q\right)x,x\rangle  \geq 
 \langle g(A^p)x,x\rangle^{1/p} \langle g(B^q)x,x\rangle^{1/q} 
\end{equation}
for all $x \in \mathcal{H}$. 
\end{theorem}

\begin{proof}
Since $g$ is an operator geodesically concave function, we have
\begin{equation}\label{theorem2.1_ineq03}
g(A\sharp_vB) \geq g(A) \nabla_v g(B) \geq g(A) \sharp_v g(B).
\end{equation}
Replacing $A,B$ with $A^p$,$B^q$, respectively, and putting $v=1/q$, then we obtain the desired inequality \eqref{theorem2.1_ineq01}.
From the first inequality of \eqref{theorem2.1_ineq03} and the arithmetic-geometric mean inequality, we have for all $x \in \mathcal{H}$, 
\begin{equation*}
\langle g\left(A^p\sharp_{1/q} B^q\right)x,x\rangle  \geq 
\frac{1}{p}\langle g(A^p)x,x\rangle  +\frac{1}{q}\langle g(A^q)x,x\rangle 
\geq \langle g(A^p)x,x\rangle^{1/p} \langle g(B^q)x,x\rangle^{1/q}.
\end{equation*}
\end{proof}
 
By letting $g(t)=1-t$ on $(0,1)$,  we have the following result.

\begin{corollary}\label{cor2.1}
Let $1/p+1/q=1$ with $p,q >1$. For commuting positive invertible operators $A$ and $B$ with spectra contained in $(0,1)$.
\begin{equation*}
1-\|(AB)^{1/2} x\|^2 \geq \left(1-\|A^{p/2}x\|^2\right)^{1/p} \left(1-\|B^{q/2}x\|^2\right)^{1/q}, 
\end{equation*}
for all $x \in \mathcal{H}$ with $||x||=1$. 
\end{corollary}

\section{Operator convex-log functions}

As it is stated in Definition \ref{definition1.1}, 
a function $f: (0, \infty) \rightarrow \mathbb{R}$ is called convex-log, if it can be written on the form $f(t)= h (\log t), \;t>0$ where $h: \mathbb{R}\rightarrow \mathbb{R}$ is a convex function. In this section we are going to present  the corresponding definition for operator functions and investigate some properties of that. 
\begin{definition}\label{sec3_def3.1}
Let $J\subset (0,\infty)$, $J_1,J_2\subset \mathbb{R}$.
We call a function $f:J\rightarrow \mathbb{R}$  operator convex-log, if it can be written on the form $f(t)= h(\log t), \;t >0$ where $h:  J_1\rightarrow J_2$ is an operator convex function.
Also, a function $g: J\to \mathbb{R}$ is said operator concave-log function, if it can be written on the form $g(t)= \varphi(\log t), \;t >0$ where $\varphi: J_1\rightarrow J_2$ is an operator concave function.
\end{definition}
\begin{remark} \label{remark3.2}
\begin{enumerate}
\item[(i)] We set $J\subset (0,\infty)$ and $J_1=J_2=(0,\infty)$ in Definition \ref{sec3_def3.1}. It is known that the function $\log(t)$ is operator concave on $(0, \infty)$. Further, every operator concave function $\varphi:(0,\infty)\rightarrow(0,\infty)$ is also operator monotone \cite{Hiai2010}. By applying these facts to the operator concave-log function $g$ we have    
\begin{align*}
g(A\nabla_v B)&= \varphi(\log (A\nabla_v B))\nonumber\\
&\geq  \varphi(\log (A)\nabla_v \log(B))\nonumber\\
&\geq \varphi(\log (A))\nabla_v \varphi(\log(B))\nonumber\\
&=g(A)\nabla_v g(B).\nonumber
\end{align*}
This means any operator concave-log function  $g: J\subset(0, \infty) \rightarrow(0,\infty)$ is an operator concave function.
\item[(ii)] There is a wide range of this class of functions. The simplest examples are the functions $(\log(t))^p$ on $ [1, \infty)$, which for $p\in [-1, 0]\cup[1, 2]$ are operator convex-log and for $p\in [0, 1]$ are operator concave-log.
\end{enumerate}
\end{remark}

In the rest of this section, we will use the following definition considred with gentle restrictions on the domains. These restrictions enable us to provide some results on the operator log-convex functions involving operator means.

\begin{definition}
We say a function $f: [1, \infty) \rightarrow [0,  \infty)$ is operator convex-log, if it can be written on the form $f(t)= h(\log t), \;t \geq1$ where $h:  [0, \infty)\rightarrow [0,  \infty)$ is an operator convex function. Similarly, a function $g:  [1, \infty) \rightarrow [0,  \infty)$ is said operator concave-log function, if it can be written on the form $g(t)= \varphi(\log t), \;t \geq1$ where $\varphi: [0, \infty)\rightarrow [0,  \infty)$ is an operator concave function.
\end{definition}
In what follows, the capital letters $A, B$
means $n\times n$ matrices or bounded linear operators on an $n$-dimensional complex Hilbert space
$\mathcal{H}$. 
For positive operators $A$ and $B$, the weak majorization  $A \prec_{w } B$ means that
 \begin{align*}
\sum_{j=1}^k \lambda_j (A) \leq \sum_{j=1}^k \lambda_j (B), \;\;\;\;\;\;\; k=1,2,\cdots, n,
 \end{align*}
 where $\lambda_1 (A) \geq \lambda_2 (A) \geq  \cdots  \geq \lambda_n (A)$ are the eigenvalues of $A$ listed in decreasing order. If equality holds when $k = n$, we have the majorization $A \prec B$.  See \cite{B1997} for more details. Also, the notation  $\preceq_{ols}$ is used for the so called Olson order. For positive operators, $A\preceq_{ols} B$ if and only if $A^r \leq B^r$ for every $r \geq 1$ \cite{O1971}. 
\begin{lemma} {\bf (\cite[Corollary 2.3]{AH1994})} \label{Lemma 3.3}
Let $A$ and $B$ be positive definite operators acting on a Hilbert space of finite dimension. Then for every $v \in [0, 1]$  
\begin{align*} \label{}
\log(A  \sharp_v B) \prec \log(A) \nabla_ v \log(B).
\end{align*}
\end{lemma}
\begin{theorem}\label{theorem2}
Let $f:[1,\infty)\rightarrow [0,\infty)$ be an  operator convex-log function. Then for every $A, B > I$ and $v \in [0, 1]$ 
\begin{align} \label{theorem2_ineq01}
f( A \sharp_v B) \prec_w f(A)  \nabla_ v  f(B).
\end{align}
\end{theorem}
\begin{proof}
Since $f$ is an operator convex-log, then there is an operator convex function $h:[0,\infty)\rightarrow[0,\infty)$ such that $f(t)= h(\log t),\; t\geq1$. Since $A, B >I$, so $A \sharp_v B > I$ and we have
\begin{align*} \label{}
f( A \sharp_v B) 
&=h(\log ( A \sharp_v B) ) \nonumber\\
&\prec_w h (\log ( A) \nabla_ v \log( B) )  \hspace*{1cm} \text{(by Lemma \ref{Lemma 3.3})}\nonumber\\
&\leq h (\log ( A)) \nabla_ v  h (\log( B) )  \hspace*{1cm} \text{(by op. convexity of $h$)} \nonumber\\
&= f(A) \nabla_ v  f(B).\nonumber
\end{align*}
In the second inequality we use the fact for every convex function $h$, $A \prec B$ implies $h(A) \prec_w h(B)$ \cite[Proposition 4.1.4]{Hiai2010}.
\end{proof}
\begin{remark}
The inequality \eqref{theorem2_ineq01} can be considered as a variant of operator geodesically convexity property for expansive operators. Also, it provides an elegant extension of Lemma \ref{Lemma 3.3}.
\end{remark}

In the sequel, we use the notation
$\mu(s,t):=\max\{S(s),S(t)\}$
where  $S(t) = \dfrac{t^{\frac{1}{t-1}}}{e \log (t^{\frac{1}{t-1}})}$  
for  $t>0$ is the so called Specht's ratio. Note that $\lim\limits_{t \rightarrow 1}S(t) = 1$ and $S(t) = S(1/t) > 1$ for $t \neq 1,\; t > 0$. For more details, see \cite{FMPS2005}.
We first give a reverse of Lemma \ref{Lemma 3.3} and then we apply it to show the next main result. The following lemmas are needed.
\begin{lemma} {\bf (\cite[Lemma 1]{GK2016})} \label{lemma 3.1_ineq01}
 Let $ 0 < s A \leq B \leq t A$, $0<s\leq t $ and $\nu\in [0, 1]$. Then 
\begin{align} \label{s}
A \nabla_\nu B \leq  \mu(s,t) ( A \sharp_\nu B ),
\end{align}
\end{lemma}

\begin{lemma} {\bf (\cite[Theorem 1]{GKF2018})} \label{lemma 3.7}
Let $H$ and $K$ be Hermitian matrices such that $ e^s e^H \preceq_{ols} e^K \preceq_{ols} e^t e^H$  for
some scalars $s \leq t$, and $v \in [0 , 1]$. Then for all $r>0$ and $k =1,2,\ldots, n$
\begin{align*} \label{}
  \lambda_k (e^{(1-v)H + v K} )  \leq  \mu^{\frac{1}{r}}\left(e^{rs},e^{rt}\right) \lambda_k (e^{r H} \sharp_v e^{r K})^{\frac{1}{r}},
\end{align*}
where $\preceq_{ols}$ is the so called Olson order. 
\end{lemma}
\begin{lemma}{\bf (\cite[Lemma 1]{GKF2018})} \label{lemma 3.8}
Let $A$ and $B$ be positive definite matrices such that $ s A \leq B \leq t A$ for some scalars $0<s\leq t$ and $v \in [0 , 1]$. Then
\begin{align*}
A^r \sharp_v B^r \leq \mu^r(s,t) (A \sharp_v B)^r\hspace*{1cm} 0<r\leq 1.
\end{align*}
\end{lemma}
\begin{proposition}\label{sec3_prop01}
Let $A$ and $B$ be positive  definite matrices such that $ e^s A \preceq_{ols} B \preceq_{ols} e^t A$  for some scalars $s \leq t$, and $v \in [0 , 1]$. Then
\begin{align} \label{e2}
 \lambda_k \big( \log A\; \nabla_ v \log B \big) 
\leq\lambda_k  \Big(\log \big( MN ( A \sharp_v B)\big)\Big),
\end{align}
and so 
\begin{align*} \label{}
\log A\; \nabla_ v \log B \prec_w \log \big( MN ( A \sharp_v B)\big),
\end{align*}
where $M:= \mu^{\frac{1}{r}}\left(e^{rs},e^{rt}\right) $, $N:= \mu\left(e^s,e^t\right)$,  and $0<r \leq 1$.
\end{proposition}
\begin{proof}
Considering the condition $ e^s A \preceq_{ols} B \preceq_{ols} e^t A$  in the form of $ e^s e^{\log A} \preceq_{ols} e^{\log B} \preceq_{ols} e^t e^{\log A}$, we can apply  Lemma \ref{lemma 3.7} by setting $H = \log A$, $K = \log B$, $M= \mu^{\frac{1}{r}}\left(e^{rs},e^{rt}\right) $ and $r>0$ as follows
\begin{align} \label{e3}
  \lambda_k \big(e^{(1-v)\log A + v \log B} \big)  
\leq  M  \lambda_k (A^r \sharp_v B^r)^{\frac{1}{r}}.
\end{align}
On the other hands, since he sandwich condition $ e^s A \preceq_{ols} B \preceq_{ols} e^t A$ implies $ e^s A \leq B \leq e^t A$, we can use Lemma \ref{lemma 3.8} for $0<r\leq 1$ as follows:
 \begin{align} \label{}
A^r \sharp_v B^r \leq \mu^{r}\left(e^s,e^t\right)  (A \sharp_v B)^r. \nonumber
\end{align}
So
 \begin{align} \label{}
   \lambda_k (A^r \sharp_v B^r)^{\frac{1}{r}}
  &\leq \Big(\mu^{r}\left(e^s,e^t\right) \lambda_k (A \sharp_v B)^r \Big)^{\frac{1}{r}}\nonumber\\
  &= \mu\left(e^s,e^t\right)  \lambda_k (A \sharp_v B).\label{e4}
\end{align}
Let $N=\mu\left(e^s,e^t\right)$. Combining the inequalities \eqref{e3} and \eqref{e4} implies 
\begin{align*} \label{}
 \lambda_k \big(e^{(1-v)\log A + v \log B} \big) 
\leq  MN \lambda_k \big(  ( A \sharp_v B)\big)=   \lambda_k \big( MN ( A \sharp_v B)\big). 
\end{align*}
Thereupon
\begin{align*} \label{}
 \log \Big( \lambda_k \big(e^{(1-v)\log A + v \log B} \big) \Big) 
\leq \log\Big( \lambda_k \big( MN ( A \sharp_v B)\big)\Big),\nonumber
\end{align*}
and hence
\begin{align*} \label{}
 \lambda_k \big( (1-v)\log A + v \log B \big) 
\leq\lambda_k  \Big(\log \big( MN ( A \sharp_v B)\big)\Big).
\end{align*}
\end{proof}
\begin{theorem}\label{theorem 3.10}
Let $g:[1,\infty)\rightarrow [0,\infty)$ be an operator concave-log  function, $A$ and $B$ be positive matrices such that $ e^s I \prec_{ols} e^s A \preceq_{ols} B \preceq_{ols} e^t A$  for some scalars $0< s \leq t$, and $v \in [0 , 1]$. Then for every $0<r \leq 1$ and $k =1,2,\ldots, n$ we have
\begin{align} \label{theorem3.2_ineq01}
 \lambda_k \big( g(A)\nabla_ v  g(B) \big) 
\leq S(e^{r t})^{\frac{1}{r}} S(e^{t}) \lambda_k  \big(g( A \sharp_v B) \big).
\end{align}

\end{theorem}
\begin{proof}
Since $g$ is an operator concave-log function, then there is an operator concave function $\varphi: [0, \infty) \rightarrow [0, \infty) $ such that $g(t)= \varphi(\log t), \; t\geq 1$. Also, according to  Remark \ref{remark3.2}, $g$ is an operator concave function. On the other hand, the sandwich condition $ e^s I \prec_{ols} e^s A \preceq_{ols} B \preceq_{ols} e^t A$ with $0<s<t$ implies $A, B>I$. Compute
\begin{align*} 
\lambda_k \big( g(A)\nabla_ v  g(B) \big) 
&= \lambda_k \big( \varphi (\log ( A)) \nabla_ v \varphi (\log( B) )\big) \nonumber\\
&\leq \lambda_k \Big(\varphi \big(\log ( A) \nabla_ v \log( B) \big)\Big)\hspace*{1cm} \text{( op. concavity of $\varphi$)}\nonumber\\
&= \varphi \Big( \lambda_k \big(\log ( A) \nabla_ v \log( B)\big)\Big)\nonumber\\
&\leq \varphi  \Big( \lambda_k \big(\log MN ( A \sharp_v  B)\big)\Big)\hspace*{1cm} \text{( \eqref{e2} and monotonicity of $\varphi$)}\nonumber\\
&= \lambda_k   \Big( \varphi \big(\log MN ( A \sharp_v  B)\big)\Big)\nonumber\\
&=  \lambda_k   \Big( g \big(MN ( A \sharp_v  B)\big)\Big)\nonumber\\
&\leq  MN \lambda_k   \Big( g \big(  A \sharp_v  B \big)\Big),\hspace*{1cm} \text{(concavity of g)}\nonumber
\end{align*}
where constants $M$ and $N$ are defined in Proposition \ref{sec3_prop01}. On the other hand, since  $S(h)$ is an increasing function on $[1, \infty)$ and $1 < e^{s} \leq e^{t}$ for $0<s\leq t$, therefore $MN=S(e^{rt})^{\frac{1}{r}} S(e^{t})$ as desired. For the last inequality, given that $MN\geq1$ we use the fact for every nonnegative concave function
  $g$ and every $z>1$, $g (zx) \leq z g(x)$.
\end{proof}
\begin{remark}
Under the assumptions of Theorem \ref{theorem 3.10} we immediately have 
\begin{align*} \label{}
 g(A)  \nabla_ v g(B)  \prec_w \mu\; g( A \sharp_v B),
\end{align*}
where $\mu=  S(e^{r t })^{\frac{1}{r}} S(e^{t})$. This inequality provides a variant of the geodesically concavity property 
 \begin{align*} \label{}
 g(a)  \nabla_ v  g(b) \leq g( a \sharp_v b),
\end{align*}
for operator concave-log functions. Also, the inequality \eqref{theorem3.2_ineq01}    is equivalent to  the existence of a unitary operator $U$ satisfying
\begin{align} \label{remark3.1_ineq01}
  g(A)   \sharp_v  g(B)  \leq g(A)  \nabla_ v g(B)  \leq \mu\; U g( A \sharp_v B) U^*.
\end{align}
\end{remark}
By applying Theorem \ref{theorem 3.10} we can get a variant of operator Acz\'{e}l  inequality involving operator concave-log functions as follows:

\begin{corollary}
Let $g:[1,\infty)\rightarrow [0,\infty)$  be an operator concave-log  function, $\frac{1}{p}+ \frac{1}{q}=1, p,q>1$ and  $ e^s I \prec_{ols} e^s A^p \preceq_{ols} B^q \preceq_{ols} e^t A^p$  for some scalars $0< s \leq t$. Then, there is a unitary operator $U$ such that for all  $x \in \mathcal{H}$
\begin{align*} 
&g (A^p) \sharp_\frac{1}{q} g( B^q)  \leq \mu\cdot  U g(A^p\sharp_\frac{1}{q} B^q) U^* , \nonumber\\ 
& \langle g(A^p)Ux, Ux\rangle^{1/p} \langle g(B^q)Ux, Ux\rangle^{1/q} \leq \mu\cdot\langle g\left(A^p\sharp_{1/q} B^q\right) Ux, Ux\rangle, \nonumber
\end{align*}
where  $\mu :=  S(e^{r t })^{\frac{1}{r}} S(e^{t})$ and $0<r \leq 1$.
\end{corollary}
\begin{proof}
Putting $A:=A^p$, $B:=B^q$ and $\nu:=1/q$ in the inequality \eqref{remark3.1_ineq01}, we have the first alleged inequality. For the second,  we first note that  the condition $ e^s I \prec_{ols} e^s A^p \preceq_{ols} B^q \preceq_{ols} e^t A^p$ implies $ e^s I \leq e^s A^p \leq B^q \leq e^t A^p$. So, by applying Lemma \ref{lemma 3.1_ineq01} for the operators $A^p$ and $B^q$ we will get
\begin{align} \label{corollary3.1_ineq03}
A^p \nabla_\nu B^q \leq  \mu\left(e^s,e^t\right) ( A^p \sharp_\nu B^p ) = S(e^t)  ( A^p \sharp_\nu B^p )\leq \mu\cdot ( A^p \sharp_\nu B^p ).
\end{align}
As it is shown in Remark \ref{remark3.2}, $g$ is an operator concave function. Also, it is  composition of two operator monotone functions. So, we can write
\begin{eqnarray*}
\mu \cdot\langle g\left(A^p\sharp_{1/q} B^q\right) Ux, Ux\rangle &\geq&  \langle g\left(\mu (A^p\sharp_{1/q} B^q)\right)Ux, Ux\rangle \hspace*{1cm}(\text{concavity of $g$})\\
&\geq&   \langle g(A^p\nabla_{1/q} B^q) Ux, Ux\rangle \hspace*{1cm}(\text{op. monotonicity of}\,\, g \;\text{with} \,\,\eqref{corollary3.1_ineq03})\\
&\geq& \langle \big(\frac{1}{p}g(A^p)+\frac{1}{q}g(B^q)\big) Ux, Ux \rangle \hspace*{1cm} (\text{op. concavity of}\,\,g)\\
&=& \frac{1}{p} \langle g(A^p) Ux, Ux\rangle + \frac{1}{q} \langle g(A^q)Ux, Ux\rangle \\
&\geq& \langle g(A^p)Ux, Ux\rangle^{1/p} \langle g(B^q)Ux, Ux\rangle^{1/q} \hspace*{1cm} (\text{AM-GM inequality}).
\end{eqnarray*}
\end{proof}

\section{ Another type of geodesically convex function }

\begin{definition} {\bf (\cite{H2019})}
A function $F : B(\mathcal{H})^+\rightarrow  \mathbb{R}$  defined in the set $B(\mathcal{H})^+$ of
positive definite operators on a finite dimensional Hilbert space $\mathcal{H}$ is said to
be a geodesically convex if
\begin{align}\label{definition 4.1}
F(A\sharp_v B)  \leq F(A) \nabla_v F(B),\quad (v \in [0,1]).
\end{align}
\end{definition}
The functions $F(A)=tr(e^A)$,  $F(A)=tr(A^\alpha),  \alpha\geq 1$, $\lambda_1(e^A)$ and $\lambda_1(A^\alpha), \alpha\geq1$ are examples of  geodesically convex functions. For more results and  examples, see \cite{SH2015}.

Bourin and Hiai in \cite[Proposition 3.5]{BH2014} showed that for every $A, B > 0$, $v \in [0, 1]$ and $k=1,2,\cdots, n$, 
\begin{align}\label{e1}
 \prod_{j=1}^k \lambda_j (A\sharp_v B) \leq \Big\lbrace \prod_{j=1}^k \lambda_j (A) \Big\rbrace \sharp_v  \Big\lbrace \prod_{j=1}^k \lambda_j (B)\Big\rbrace,
\end{align}
and 
\begin{align}\label{e6}
 \prod_{j=n+1-k}^n \lambda_j (A\sharp_v B) \geq \Big\lbrace \prod_{j=n+1-k}^n \lambda_j (A) \Big\rbrace \sharp_v\Big\lbrace\prod_{j=n+1-k}^n \lambda_j (B) \Big\rbrace.
\end{align}
It is deduced from the inequality \eqref{e1} that  $F(A)= \prod\limits_{j=1}^k \lambda_j (A)$ and $F(A)= \det(A)$ are also geodesically convex  functions.
In this section, we investigate geodesically convexity property of some new functions involved with operator functions and achieve  generalization of the above Minkowski type inequalities, simultaneously.

It is shown in \cite[Theorem 2.3]{SH2015} if $ h $ is an increasing convex function on $(0, \infty)$, then $\sum\limits_{j=1}^k h(\lambda_j (A))$ is geodesically convex. In the following, we give a corresponding result for  increasing geometrically convex functions on $(0, \infty)$.

\begin{lemma}\label{lemma4.2}
Let $ g $ be an increasing geometrically convex function on $(0, \infty)$. Then the function $F(A)= \sum\limits_{j=1}^k g(\lambda_j (A))$, $k=1,2,\cdots, n$, is geodesically convex function.
\end{lemma}
\begin{proof}
First, note that the inequality \eqref{e1} is equivalent to the following one
\begin{align}\label{lemma4.2_ineq01}
 \prod_{j=1}^k \lambda_j (A\sharp_v B) \leq  \prod_{j=1}^k \lambda_j (A) ^v  \lambda_j (B)^{1-v} = \prod_{j=1}^k \lambda_j (A^v)   \lambda_j (B^{1-v}).
\end{align}
Also, since $g(t)$ is  a  geometrically convex function, $g(e^t)$ is a convex function due to the following inequality
\begin{equation*}
g(e^{x \nabla_v y})= g(e^{(1-v)x} e^{v y})= g(e^x  \sharp_v e^y)\leq g(e^x)  \sharp_v g(e^y) \leq g(e^x) \nabla_v g(e^y).
\end{equation*}
Now, by applying a  classical result  on the function $g$ \cite[Proposition 4.1.6]{Hiai2010} and the inequality \eqref{lemma4.2_ineq01}  we have
\begin{align*}
 \sum_{j=1}^k g\big(\lambda_j (A\sharp_v B) \big) \leq \sum_{j=1}^k g\big(\lambda_j ^{1-v}( A ) \lambda_j^v(B)\big).
\end{align*}
Hence, we can write
\begin{align*}
 \sum_{j=1}^k g\big(\lambda_j (A\sharp_v B) \big)
 &\leq \sum_{j=1}^k g\big(\lambda_j ^{1-v}( A ) \lambda_j^v(B)\big)\nonumber\\
 &\leq \sum_{j=1}^k \big( g(\lambda_j ( A ))\big)^{1-v} g\big((\lambda_j(B))\big)^v\hspace*{1cm} \text{(g. convexity of $g$)}\nonumber\\
 &\leq\Big\lbrace \sum_{j=1}^k g\big(\lambda_j ( A )\big)\Big\rbrace \sharp_v \Big\lbrace \sum_{j=1}^k g\big(\lambda_j(B)\big)\Big\rbrace  \hspace*{1cm} \text{(Cauchy-Schwarz inequality)} \nonumber\\
 &\leq\Big\lbrace \prod_{j=1}^k g\big(\lambda_j ( A )\big)\Big\rbrace  \nabla_v\Big\lbrace \prod_{j=1}^k g \big(\lambda_j(B)\big)\Big\rbrace, \hspace*{1cm} \text{(AM-GM inequality)}.\nonumber
\end{align*}
\end{proof}

In the sequel, we present some results involving operator functions.
\begin{lemma}{\bf (\cite{KF2018})}\label{l1}
 Let $ g $ be a nonnegative operator monotone decreasing function on $(0, \infty)$ and $0 < s A \leq  B  \leq t A$ for some constants $0 < s \leq t$. Then, for all $v \in [0 , 1]$
\begin{align*} 
 g( A \sharp_v B) \leq  \mu(s,t) (g(A)\sharp_v g(B)).
\end{align*}
\end{lemma}

\begin{theorem}\label{t1}
Let $g$ be a nonnegative  operator monotone decreasing function on $(0, \infty)$ and $ 0 < s A \leq B \leq t A$ for some scalars $s , t > 0 $. Then for all $v \in [0, 1]$ and $k=1,2,\cdots, n$,
\begin{align*}
 \prod_{j=1}^k \lambda_j (g(A\sharp_v B)) \leq \mu^k\left(s,t\right)  \Big( \Big\{ \prod_{j=1}^k  \lambda_j (g(A)) \Big\}\sharp_v \Big\{  \prod_{j=1}^k  \lambda_j (g(B)) \Big\} \Big).
\end{align*}
\end{theorem}
\begin{proof}
We compute
\begin{align}
 \prod_{j=1}^k \lambda_j (g(A\sharp_v B))
 & \leq  \prod_{j=1}^k \lambda_j  \big( \mu(s,t) \left(g(A) \sharp_v g(B)\right)\big) \hspace*{1cm} \text{(by Lemma \ref{l1})}\nonumber\\
  & = \mu^k(s,t)  \prod_{j=1}^k \lambda_j  ( g(A) \sharp_v g(B)) \nonumber\\
 & \leq  \mu^k(s,t)  \Big(\Big\{ \prod_{j=1}^k \lambda_j (g(A)) \Big\} \sharp_v \Big\{ \prod_{j=1}^k \lambda_j (g(B)) \Big\}\Big) \hspace*{1cm} \text{(by \eqref{e1})}.\nonumber
\end{align}
\end{proof}
\begin{corollary}
Let $g$ be a nonnegative  operator monotone decreasing function on $(0, \infty)$ and $ 0 < s A \leq B \leq t A$ for some scalars $s, t>0$. Then for all $\nu \in [0, 1]$ 
\begin{align*}
\det g(A\sharp_v B)  \leq \mu^n(s,t)  \left( \det g(A) \sharp_v  \det g(B) \right).
\end{align*}
\end{corollary}
\begin{remark}
According to Theorem \ref{t1}, by letting $F(A)= \prod\limits_{j=1}^k \lambda_j (g(A))$ where $g$ is operator monotone decreasing function on $(0, \infty)$, we have
\begin{align*}
F(A\sharp_v B)  \leq  \mu^k(s,t)   \big(F(A) \nabla_v F(B)\big).
\end{align*}
This inequality gives a variant of geodesically convexity property of \eqref{definition 4.1} for the function $F(A)=\prod\limits_{j=1}^k \lambda_j (g(A))$. Further, it provides an extension of  Minkowski type inequality \eqref{e1} to the operator functions.
\end{remark}
In the next, we will see an extension of  Minkowski type inequality \eqref{e6}.
\begin{theorem}
Let $f$ be an operator monotone function on $(0,\infty)$ and $ 0 < s A \leq B \leq t A$ for some scalars $s , t > 0 $. Then for all $v \in [0, 1]$ and $k=1,2,\cdots, n$,
\begin{align*}
 \prod_{j=n+1-k}^n \lambda_j (f(A\sharp_v B)) \geq \mu^k(s,t)   \Big( \Big\{ \prod_{j=n+1-k}^n \lambda_j (f(A)) \Big\} \sharp_v \Big\{ \prod_{j=n+1-k}^n  \lambda_j (f(B)) \Big\} \Big)
\end{align*}
\end{theorem}
\begin{proof}
Since $f$  is operator monotone  on $(0,\infty)$, so $1/ f$ is  operator monotone decreasing on $(0,\infty)$ and we can apply Theorem \ref{t1} for $g=1/ f$ as follows.
\begin{align*}
 \prod_{j=1}^k \lambda_j \big((f(A\sharp_v B))^{-1}\big) \leq  \mu^k(s,t)  \Big\{ \prod_{j=1}^k \lambda_j \left((f(A))^{-1}\right) \Big\} \sharp_v \Big\{ \prod_{j=1}^k \lambda_j \left((f(B))^{-1}\right) \Big\},
\end{align*}
and hence
\begin{align*}
 \prod_{j=n+1-k}^n \left(\lambda_j \big(f(A\sharp_v B)\big)\right)^{-1} \leq   \mu^k(s,t)  \Big\{ \prod_{j=n+1-k}^n \left(\lambda_j \left(f(A)\right)\right)^{-1}\Big\} \sharp_v \Big\{ \prod_{j=n+1-k}^n \left(\lambda_j \left(f(B)\right)\right)^{-1} \Big\}.
\end{align*}
By using the property $X^{-1} \sharp_v Y^{-1}= (X \sharp_v Y)^{-1}$ and reversing the inequality, we get the desired result.
\end{proof}
\begin{corollary}
Let $f$ be an operator monotone function on $(0,\infty)$ and $ 0 < s A \leq B \leq t A$ for some scalars $s , t > 0 $. Then for all $v \in [0, 1]$
\begin{align*}
\det f(A\sharp_v B)  \geq \mu^n(s,t)  \left( \det f(A) \sharp_v  \det f(B) \right).
\end{align*}
\end{corollary}
\begin{remark}
The appeared constant $\mu(s,t)=\max \lbrace  S(s), S(t)\rbrace$ in the all results of the preceding sections can be replaced by $\max \lbrace  K(s)^R, K(t)^R \rbrace$ where $K(h) = \dfrac{(h+1)^2}{4h}, h>0$ is the Kantorovich constant and  $R= \max \lbrace v, 1-v \rbrace$, with no ordering between them. See \cite{KF2018}.
\end{remark}

\section*{Acknowledgements}
The authors (M.R.J.) and (V.K.) were supported by Iran National Science Foundation (INSF), Project Number 96009632, and the author (S.F.) was partially supported by JSPS KAKENHI Grant Number 16K05257.

  \bibliographystyle{elsarticle-num} 


\end{document}